\documentclass[12pt,bezier]{article}
\usepackage{times}
\usepackage{booktabs}
\usepackage{pifont}
\usepackage{floatrow}
\usepackage{caption}
\usepackage{mathrsfs}
\usepackage[fleqn]{amsmath}
\usepackage{amsfonts,amsthm,amssymb,mathrsfs,bbding}
\usepackage{txfonts}
\usepackage{graphics,multicol}
\usepackage{graphicx}
\usepackage{color}
\usepackage{caption}
\usepackage{cite}
\usepackage{latexsym,bm}
\usepackage{indentfirst}
\evensidemargin=\oddsidemargin\topmargin=-1.5cm

\newtheorem{lemma}{Lemma}[section]
\newtheorem{theorem}{Theorem}[section]

\newtheorem{que}{Question}[section]

\newtheorem{defi}{Definition}[section]
\theoremstyle{definition}

\addtocounter{section}{0}

\title{{\bf The maximum number of $s$-cliques in connected graphs and its application to spectral moment}
\thanks{Supported by the National Natural Science Foundation of China (Nos. 11971445 and 11871222), SRSF of Chuzhou University (No. 2018qd02).}}

\author{{\bf Longfei Fang}, {\bf Mingqing Zhai}\thanks{Corresponding author. E-mail addresses: 1053788649@qq.com (L. Fang); mqzhai@chzu.edu.cn (M.
Zhai); wuyuwuyou@126.com (B. Wang).}, {\bf Bing Wang}
\\
{\footnotesize  School of Mathematics and Finance, Chuzhou
University, Anhui, Chuzhou, 239012, China}}
\date{}

\begin{document}
\openup 1.0\jot \maketitle

\begin{abstract}
Extremal problems concerning the number of complete subgraphs have a long story in extremal graph theory. Let $k_s(G)$ be the number of $s$-cliques in a graph $G$ and $m={{r_m}\choose s}+t_m$, where $0\le t_m\leq r_m$.
Edr\H{o}s showed that $k_s(G)\le {{r_m}\choose s}+{{t_m}\choose{s-1}}$ over all graphs of size $m$ and order $n\geq r_m+1$.
%Clearly, $K_{r_m}^{t_m}\cup (n-r_m-1)K_1$ is an extremal graph, where $K_{r_m}^{t_m}$ is the graph by joining a new vertex to $t_m$ vertices of $K_{r_m}$.
It is natural to consider an improvement in connected situation: what is the maximum number of $s$-cliques over all connected graphs of size $m$ and order $n$? In this paper, the sharp upper bound of $k_s(G)$ is obtained and extremal graphs are completely characterized. The technique and the bound are different from those in general case. As an application, this result can be used to solve a question on spectral moment.


\end{abstract}

\section{Introduction}

Graphs considered here are undirected, finite and simple. $V(G)$ and $E(G)$ are the vertex set and edge set of a graph $G$, respectively. For a vertex $u\in V(G) $, we denote the neighborhood of $u$ by $N_G(u)$ and the closed neighborhood of $u$ by $N_G[u]$. Let $G[S]$ denote the subgraph of $ G $ induced by a vertex subset $S$. For a vertex $v\in V(G)$, let $G-v $ denote the induced subgraph $G[V(G)\setminus \{v\}]$. Let $P_k$, $C_k$ $K_k$ be a path, a cycle and a complete graph of order $k$, respectively.  An \emph{$s$-clique} is a complete subgraph of order $s$ and the number of $s$-cliques in $G$ is denoted by $k_s(G)$.

The problem of determining $k_s(G)$ is a very interesting topic in extremal graph theory.
On one hand, a number of papers investigated the lower bound of $k_s(G)$ (see for example, \cite{bca,bcb,bcc,bcd,bcf}). On the other hand, much attentions have been paid to the upper bound of $k_s(G)$. The well-known Tur\'{a}n's theorem \cite{tur} obtained $k_2(G)$ in $K_{r+1}$-free graphs. Recently, Ergemlidze, Methuku, Salia, Gy\H{o}ri \cite{ERG} obtained $k_3(G)$ in a $C_5$-free graphs. Luo \cite{LUO}
bounded $k_s(G)$ in a graph without cycles of length more than $k$.
A conjecture due to Engbers and Galvin \cite{eg} asks which graph maximizes the number of $s$-cliques over all $n$-vertex graphs with maximum degree $\Delta$ for every $s\ge 3$. This conjecture has drawn attentions of many researchers (see \cite{faa,fcc,fdd}). Recently, Chase \cite{fee} completely solved this conjecture. When substitute the maximum degree with the average degree, Edr\H{o}s \cite{bce} considered the following question: what is the maximum number of $s$-cliques over all graphs with fixed size and order?

\begin{theorem}\label{001}(\cite{bce})
Let $m,n,s$ be positive integers with $s\geq3$ and ${{s}\choose 2}\le m\le {{n}\choose 2}$. Put $m={{r_m}\choose 2}+t_m$, where $0\le t_m\leq r_m$. Then $k_s(G)\leq {{r_m}\choose s}+{{t_m}\choose {s-1}}$ for any graph $G$ of size $m$ and order $n$.
\end{theorem}

Let $K_{r_m}^{t_m}$ be the graph obtained by joining a new vertex to $t_m$ vertices of a copy $K_{r_m}$. One can see that the bound in Theorem \ref{001} is attained when $G$ is isomorphic to $K_{r_m}^{t_m}$ with possibly some isolated vertices. Naturally, one may consider an improvement by the following question.

\begin{que}\label{ques1}
Let $\mathscr{G}_{m,n}$ be the set of connected graphs of size $m$ and order $n$.
what is the maximum number of $s$-cliques over all graphs in $\mathscr{G}_{m,n}$?
\end{que}

In Section 2, the sharp upper bound in Question \ref{ques1} is obtained and corresponding extremal graphs are completely characterized.

Another motivation of considering Question \ref{ques1} comes from a classic topic in spectral graph theory.
Let $\lambda_1(G),\lambda_2(G),\cdots,\lambda_n(G)$ be the eigenvalues in nonincreasing order of the adjacency matrix $A(G)$ of a graph $G$. For an integer $j\in [0,n-1]$, $\sum_{i=1}^{n}\lambda_i^j(G)$ is called the $j$-th \emph{spectral moment} of $G$ and denoted by $S_j(G)$. We know that $S_j(G)$ is the number of closed walks of length $j$ (see \cite{aa}). Hence, $S_3(G)=6k_3(G)$. The sequence of spectral moments $(S_0(G),S_1(G),\dots,S_{n-1}(G))$ is denoted by $S(G)$. For two graphs $G_1$ and $G_2$, we denote $G_1\prec_S G_2 $ if there is a positive integer $j\le n-1$ such that $ S_i(G_1)=S_i(G_2) $ for $i<j$ and $S_j(G_1)<S_j(G_2)$. Now, we also say that $G_1$ comes before $G_2$ in $ S $-order. Up to now, many results on the $S$-order of graphs have been obtained. Andriantian and Wagner \cite{da} characterized the trees with maximum $S_k(T)$ over all trees with a given degree sequence. Cvetkovi$\acute{c}$ et al. \cite{bab} characterized the first and the last graphs in $ S $-order of trees and unicyclic graphs. Pan et al. \cite{bad} gave the last and the second last quasi-trees. Cheng et al. \cite{b} determined the last $d+\lfloor\frac{d}{2}\rfloor-2$ graphs, in $S$-order, of all trees with order $n$ and diameter $d$. Li and Zhang \cite{bag} characterized the last and the second last graphs, in $S$-order, over all graphs with given number of cut edges. As an application of main result, in Section 3, we determine the last graph in $ S $-order over all graphs in $ \mathscr{G}_{m,n}$.

%\begin{theorem}\label{123} (\cite{bbh})
%Let $k$ be an odd prime. A graph $G$ with $m$ edges and $n$ vertices has at most
%$C_{n,m}=\frac{2^{\frac{k}{2}-1}[(n-1)^{k-1}-1]m^{\frac{k}{2}}}{{kn^{\frac{k}{2}}}{{(n-1)}^{\frac{k}{2}-1}}}$
%cycles of length $k$, with equality if and only if $G\cong K_n$.
%\end{theorem}
%
%When $k=3$, Theorem \ref{123} gives a more precise bound: $k_3(G)\le \frac{\sqrt{2}(n-2)}{3\sqrt{n(n-1)}}m^{\frac{3}{2}}$. The extremal graph implies that this bound is sharp only when $m={{n}\choose2}$. Hence, Rivin proposed the following question.
%
%\begin{que}\label{ques2} (\cite{bbh})
% Consider all graphs with $m$ edges and $n$ vertices, is there a way to characterize the one with the most triangles?
%\end{que}

\section{The maximum number of $s$-cliques in connected graphs of given size and order}

First, let us introduce some definitions and notations. A maximal 2-connected subgraph in a graph is called a \emph{block}.
For a given graph $G$ and a positive integer $s$, denote by $ G^s$ the graph obtained from $G$ by iteratively removing the vertices of degree at most $s$ until the resulting graph has minimum degree at least $s+1$ or is empty. It is well known that $G^{s}$ is unique and does not depend on the order of vertex deletion (see~\cite{pit}). The following lemmas will be used to prove the main result of this section.

\begin{lemma}\label{aaa}
If $H_0$ is an induced subgraph of a connected graph $G$ with $[|E({G})|-|V({G})|]-[|E({H_0})|-|V({H_0})|]=k$. Then $k\ge 0$.
Furthermore, $H_0^{s-2}\cong G^{s-2}$ for any $s\ge k+3$.
\end{lemma}

\begin{proof} The case $H_0\cong G$ is trivial. Assume that $H_0\ncong G$ and $V(G)\setminus V(H_0)=\{v_1,v_2,\cdots,v_a\}$. Since $G$ is connected, there exists some vertex $v_i$, say $v_1$, such that $v_1$ is adjacent to some vertex of $H_0$. Define $H_1=G[V(H_0)\cup \{v_1\}]$. Repeat this step, we obtain a sequence of graphs $H_1,H_2,\cdots,H_a=G$ with $d_{H_{i-1}}(v_i)\geq1$ for each $i\in\{1,2,\ldots,a\}$. Clearly, $|E(G)|-|V(G)|=|E({H_0})|-|V({H_0})|+\sum_{i=1}^a[d_{H_{i-1}}(v_i)-1]$. So $k=\sum_{i=1}^a[d_{H_{i-1}}(v_i)-1]\geq0$. This also indicates that $d_{H_{i-1}}(v_i)\leq k+1\leq s-2$ for any $i\in \{1,2,\cdots,a\}$. Since $G^{s-2}$ does not depend on the order of vertex deletion, we have $H_0^{s-2}\cong G^{s-2}$.
\end{proof}

Now we introduce two notations $r_{m,n}$ and $t_{m,n}$, which will be frequently used in subsequent discussions.

\begin{defi}\label{de1}
Let $m,n$ be two positive integers with $n-1\leq m\leq {n\choose 2}$. Put $$m-n={{r_{m,n}-1\choose 2}}+t_{m,n}-2,  \eqno(1)$$
where $r_{m,n}$ and $t_{m,n}$ are two positive integers. Specially, $2\le t_{m,n}\le r_{m,n}$ if $m-n\ge 0$; $t_{m,n}=r_{m,n}=1$ if $m-n=-1$.
\end{defi}

%\begin{pro}\label{pr1}
%$r_{m,n}\leq n-1$ for any two positive integers $m,n$ with $n-1\leq m\leq {n\choose 2}$.
%\end{pro}
%
%\begin{proof} By Definition \ref{de1}, ${{r_{m,n}-1}\choose2} \le m-n \le {{n}\choose2}-n= {{n-1}\choose2}-1 $. Hence $r_{m,n}\leq n-1$.
%\end{proof}

\begin{lemma}\label{33}
For any graph $G\in \mathscr{G}_{m,n}$ with $m-n\ge 0$,
there exists a non-cutvertex $u_0$ of $G^1 $ such that $d_{G^1}(u_0) \le r_{m,n}-1$, unless $G^1\cong K_{r_{m,n}+1}$.
\end{lemma}

\begin{proof}
Let $G^1\in \mathscr{G}_{m',n'}$ and $G^1\ncong K_{r_{m,n}+1}$.
Suppose to the contrary that $d_{G^1}(u)\ge r_{m,n}$ for any non-cutvertex $u$ of $G^1$.
If $G^1$ has cutvertices, then there is at least two blocks $B_i$ containing exactly one cutvertex $u_i$ of $G^1$ for $i\in\{1,2\}$. Clearly, $$|E(B_i)|-n_i\ge \frac 12 [(n_i-1)r_{m,n}+d_{B_i}(u_i)]-n_i\ge\frac{(n_i-1)(r_{m,n}-2)}2\ge\frac{r_{m,n}(r_{m,n}-2)}2,$$ where $d_{B_i}(u_i)\geq2$ and $n_i=|V(B_i)|\geq r_{m,n}+1$ for $i\in\{1,2\}$. Note that $|E(B)|\geq |V(B)|$ for each block $B$ of $G^1$.
Then $$m'-n'> \sum_{i=1}^2(|E(B_i)|-n_i)\ge r_{m,n}^2-2r_{m,n}\ge m-n,$$
which contradicts Lemma \ref{aaa}.
So $G^1$ itself is a block. Since $d_{G^1}(u)\ge r_{m,n}$ for any non-cutvertex $u$ of $G^1$,
we have $n'\geq r_{m,n}+1$ and
$$m'-n'\ge\frac 12n'r_{m,n}-n'\ge\frac{(r_{m,n}+1)(r_{m,n}-2)}2={{r_{m,n}-1}\choose2 }+(r_{m,n}-2)\geq m-n.$$
Since $G^1\ncong K_{r_{m,n}+1}$, the first two inequalities of above can not be equalities simutaneously.
Hence, we also have $m'-n'>m-n$, a contradiction. This completes the proof.
\end{proof}

\begin{lemma}\label{aab}
For any integers $a\ge b\ge 1$ and $s\ge 2$, ${a\choose s}+{b\choose s}\le {{a+1}\choose s}+{{b-1}\choose s}$. More generally, for any integer $c$ with $\max\{a,b\}\le c \le a+b$, ${a\choose s}+{b\choose s}\le {c\choose s}+{{a+b-c}\choose s}$, with equality if and only if $c\le s-1$ or $c=\max\{a,b\}$.
\end{lemma}

\begin{proof} It is known that ${{a}\choose s}+{a\choose s-1}={a+1\choose {s}}$.
Hence, $${a\choose s}+{b\choose s}={{a+1}\choose s}-{a\choose s-1}+{{b-1}\choose s}+{b-1\choose s-1}\leq{{a+1}\choose s}+{{b-1}\choose s},$$
with equality if and only if $a\leq s-2$. Iterate to get the general result.
\end{proof}

\begin{lemma}\label{ccc}
Let $m,n,s$ be three positive integers with $s\ge 3$ and $-1\le m-n\le {s\choose 2}-s-1$.
If $G\in\mathscr{G}_{m,n}$, then $k_s(G)=0$.
\end{lemma}

\begin{proof}
Suppose to the contrary, then there is an induced subgraph $H$ of $G$ such that $H\cong K_s$. By Lemma~\ref{aaa}, we have
$${s\choose 2}-s=|E(H)|-|V(H)|\le|E({G})|-|V({G})|\le {s\choose 2}-s-1,$$ a contradiction.
\end{proof}

Now we give two main results on the maximum number of $s$-cliques in connected graphs of given size and order.

\begin{theorem}\label{111}
For any positive integers $m,n,s$ with $s\ge 3$ and $n-1\le m\leq {{n}\choose2}$, $max~\{k_s(G)~|~G\in\mathscr{G}_{m,n}\}={{r_{m,n}}\choose s}+{{t_{m,n}}\choose{s-1}}$.
\end{theorem}

\begin{proof} The proof is proceeded by induction on $m-n$. Firstly, we consider the case $-1\le m-n\le{s\choose 2}-s-1$. By Lemma \ref{ccc}, $k_s(G)=0$. And by Definition \ref{de1},
we can find that $r_{m,n}\le s-1$ and $t_{m,n}\le s-2$.
So the result holds trivially.

In the following we need to consider the case $m-n\ge{s\choose 2}-s$. Let $G^{*}$ be the graph obtained from $K_{r_{m,n}}^{t_{m,n}}$ by adding ${n-r_{m,n}-1}$ pendant edges to a vertex of degree $r_{m,n}$. Clearly, $G^*\in \mathscr{G}_{m,n}$ and $k_s(G^*)={{r_{m,n}}\choose s}+{{t_{m,n}}\choose{s-1}}$. This implies that the maximum of $k_s(G)$ is at least as large as claimed for any $m,n$ with $m-n\ge{s\choose 2}-s$. So it suffices to show $k_s(G)\leq{{r_{m,n}}\choose s}+{{t_{m,n}}\choose{s-1}}$ for any graph $G\in\mathscr{G}_{m,n}$.
Assume that $G^1\in\mathscr{G}_{m',n'}$, clearly, $m'-n'=m-n$.
If $G^1\cong K_{r_{m,n}+1}$, then $k_s(G)={{r_{m,n}+1}\choose s}={{r_{m,n}}\choose s}+{{t_{m,n}}\choose{s-1}},$ as desired.
Suppose that $G^1\ncong K_{r_{m,n}+1}$, then
by Lemma \ref{33}, there exists a non-cutvertex $u_0$ of $G^1$ with $d_{G^1}(u_0) \le r_{m,n}-1$. Denote $H=G^1-u_0$ and $d_0=d_{G^1}(u_0)$. Then $H\in \mathscr{G}_{m'-d_0,n'-1}$, where $$(m'-d_0)-(n'-1)={{r_{m,n}-1}\choose2}+(t_{m,n}-d_0+1)-2. \eqno(2)$$
And we can see that $$k_s(G)=k_s(G^1)=k_s(H)+k_{s-1}(G^1[N(u_0)]),~~~k_{s-1}(G^1[N(u_0)])\leq{{d_0}\choose{s-1}}. \eqno(3)$$
Now we distinguish two cases to complete the proof.

\noindent{\bf{Case 1.}} $2\leq d_0\le t_{m,n}-1$. Then $2\leq t_{m,n}-d_0+1\leq r_{m,n}$. Comparing equalities (1) and (2), we have $r_{m'-d_0,n'-1}=r_{m,n}$ and $t_{m'-d_0,n'-1}=t_{m,n}-d_0+1$.
And by the induction hypothesis,
$$k_s(G)\leq k_s(H)+{{d_0}\choose{s-1}}
\le{{r_{m,n}}\choose s}+{{t_{m,n}-d_0+1}\choose{s-1}}+{{d_0}\choose{s-1}}
\leq{{r_{m,n}}\choose s}+{{t_{m,n}}\choose{s-1}}, \eqno(4)$$
where the last inequality comes from setting $c=t_{m,n}$ in Lemma \ref{aab}.

\noindent{\bf{Case 2.}} $t_{m,n}\leq d_0\le r_{m,n}-1$. Now, equality (2) can be written as follows:
$$(m'-d_0)-(n'-1)={{r_{m,n}-2}\choose2}+(r_{m,n}+t_{m,n}-d_0-1)-2, \eqno(5)$$ where $2\leq r_{m,n}+t_{m,n}-d_0-1\leq r_{m,n}-1$. Comparing equalities (1) and (5), we have $r_{m'-d_0,n'-1}=r_{m,n}-1$ and $t_{m'-d_0,n'-1}=r_{m,n}+t_{m,n}-d_0-1.$
By the induction hypothesis,
$$k_s(G)\leq k_s(H)+{{d_0}\choose{s-1}}
\le {{r_{m,n}-1}\choose s}+{{r_{m,n}+t_{m,n}-d_0-1}\choose {s-1}}+{{d_0}\choose{s-1}}\leq{{r_{m,n}}\choose s}+{{t_{m,n}}\choose{s-1}}, \eqno(6)$$
where the last inequality comes from setting $c=r_{m,n}-1$ in Lemma \ref{aab}
and a well-known combinatorial identity ${{r_{m,n}-1}\choose s}+{{r_{m,n}-1}\choose s-1}={{r_{m,n}}\choose s}$.
This completes the proof.
\end{proof}

Let $\mathbb{B}(p,q)$ be the set of graphs obtained from $K_p$ and $C_q$ by adding a path of length $r\geq 0$ between one vertex of $K_p$ and another of $C_q$. Let $B(p,q)$ denote the graph in $\mathbb{B}(p,q)$ with $r=0$. By Lemma \ref{ccc}, $k_s(G)=0$ for any graph $G\in\mathscr{G}_{m,n}$ with $m-n\leq {s\choose 2}-s-1$.
The following theorem characterizes the extremal graphs for $m-n\geq{s\choose 2}-s.$

\begin{theorem}\label{222} Let $m, n,s$ be three integers with $s\geq3$ and ${s\choose 2}-s\le m-n\leq {{n}\choose2}-n$. Let $G$ be an extremal graph with maximal number of $s$-cliques over all graphs in $\mathscr{G}_{m,n}$.\\
 (\romannumeral1) If $t_{m,n}\leq s-2$, then $G^{s-2}\cong  K_{r_{m,n}}$; \\
 (\romannumeral2) if $t_{m,n}\ge s-1$, then $G^{s-2}\cong K_{r_{m,n}}^{t_{m,n}}$ or particularly, $G^{s-2} \in \mathbb{B}(r_{m,n},3)$ for $s=3$, $t_{m,n}=2$ and $r_{m,n}\geq3$.
\end{theorem}

\begin{proof} Since $m-n\geq{s\choose 2}-s$, by Definition \ref{de1}, it is easy to see that $r_{m,n}\ge s-1$ and particularly, if $r_{m,n}=s-1$ then $t_{m,n}=s-1$. We shall consider the case that $G^1\cong K_{r_{m,n}+1}$.
Now, $m-n=|E(G^1)|-|V(G^1)|={r_{m,n}+1\choose2}-(r_{m,n}+1)={r_{m,n}-1\choose2}+r_{m,n}-2.$
By Definition \ref{de1}, $t_{m,n}=r_{m,n}.$ So $t_{m,n}\ge s-1$. Since $\delta(G^1)=r_{m,n}\ge s-1$, $G^{s-2}\cong K_{r_{m,n}+1}\cong K_{r_{m,n}}^{r_{m,n}}$, as desired.

Let $u_0$ be a non-cutvertex of $G^1$ with minimal degree.
Let $G^1\in\mathscr{G}_{m',n'}$, $H=G^1-u_0$ and $d_0=d_{G^1}(u_0)$. Then $H\in\mathscr{G}_{m'-d_0,n'-1}.$
The proof is proceeded by induction on $m-n$. Firstly, assume that $m-n={s\choose 2}-s$, then $r_{m,n}=t_{m,n}=s-1$.
If $G^1\cong K_{r_{m,n}+1}$, we are done. Otherwise, by Lemma \ref{33}, $d_0\leq r_{m,n}-1=s-2$ and hence $k_{s-1}(G^1[N(u_0)])=0$. Moreover, since $m'-n'=m-n$ and $d_0\geq2$, $(m'-d_0)-(n'-1)<{s\choose 2}-s.$ By Lemma \ref{ccc}, $k_s(H)=0$. Thus, by (3), $k_s(G)=0<{{r_{m,n}}\choose s}+{{t_{m,n}}\choose {s-1}}$, a contradiction to the maximality of $k_s(G)$.
Now it suffices to consider the following two cases.

\noindent{\bf{Case 1.}} $2\leq d_0\le t_{m,n}-1$.

Recall that $r_{m'-d_0,n'-1}=r_{m,n}$, $t_{m'-d_0,n'-1}=t_{m,n}-d_0+1<t_{m,n}$.
Since $k_s(G)={{r_{m,n}}\choose s}+{{t_{m,n}}\choose {s-1}}$, all the inequalities in (4) are equalities.
By Lemma \ref{aab}, the last inequality of (4) holds in equality if and only if $t_{m,n}\leq s-2$.
So $r_{m,n}\geq s$ and ${{d_0}\choose {s-1}}=0$. Now, the first inequality of (4) holds naturally in equality.
The second inequality of (4) holds in equality if and only if $k_s(H)={{r_{m,n}}\choose s}+{{t_{m,n}-d_0+1}\choose {s-1}}$. Now $k_s(H)>0$ and hence ${s\choose 2}-s\leq(m'-d_0)-(n'-1)<m-n$. Since $t_{m'-d_0,n'-1}<t_{m,n}\leq s-2$, $H$ meets (i). By the induction hypothesis, $H^{s-2}\cong K_{r_{m,n}}$. Since $d_0\leq s-2$ and $G^{s-2}$ does not depend on the order of vertex deletion, we have $G^{s-2}\cong(G^1)^{s-2}\cong H^{s-2}\cong K_{r_{m,n}}$.

\noindent{\bf{Case 2.}} $t_{m,n}\le d_0\le r_{m,n}-1$.

Recall that $r_{m'-d_0,n'-1}=r_{m,n}-1$, $t_{m'-d_0,n'-1}=r_{m,n}+t_{m,n}-d_0-1.$ Now $r_{m,n}\ge s$ (otherwise,
$r_{m,n}=s-1$, then $t_{m,n}=s-1$, which contradicts to $t_{m,n}\le r_{m,n}-1$).
By Lemma \ref{aab}, the last inequality of (6) holds in equality if and only if $r_{m,n}-1\leq s-2$ or
$d_0\in\{t_{m,n},r_{m,n}-1\}$. Since $r_{m,n}\geq s$, we have $d_0\in\{t_{m,n},r_{m,n}-1\}$.

\noindent{\bf{Subcase 2.1.}} $d_0=t_{m,n}$.

Now $t_{m'-d_0,n'-1}=r_{m'-d_0,n'-1}=r_{m,n}-1\geq s-1$. The second inequality of (6) holds in equality if and only if
$k_s(H)={{r_{m,n}-1}\choose s}+{{r_{m,n}-1}\choose {s-1}}$. Now $k_s(H)={{r_{m,n}}\choose s}>0$ and hence ${s\choose 2}-s\leq(m'-d_0)-(n'-1)<m-n$. By the induction hypothesis, $H$ meets (ii), that is, $H^{s-2}\cong K_{r_{m,n}-1}^{r_{m,n}-1}\cong K_{r_{m,n}}$. If $d_0=t_{m,n}\le s-2$, similar to Case 1, we have $G^{s-2}\cong K_{r_{m,n}}$. Next suppose that $d_0=t_{m,n}\geq s-1.$
Clearly, $|E(H^{s-2})|-|V(H^{s-2})|={{r_{m,n}}\choose 2}-r_{m,n}$. And by equality (2), $|E(H)|-|V(H)|={{r_{m,n}-1}\choose 2}-1$. So $|E(H^{s-2})|-|V(H^{s-2})|=|E(H)|-|V(H)|$. By Lemma~\ref{aaa}, $H^1\cong(H^{s-2})^1$ and hence $H^1\cong K_{r_{m,n}}$.
Note that the first inequality of (6) holds in equality if and only if $N(u_0)$ is a $t_{m,n}$-clique.
If $N(u_0)\nsubseteq V(H^1)$, then $t_{m,n}=2$ and hence $s=3$, $r_{m,n}\geq s=3$.
This implies that $G^{s-2}=G^1\in\mathbb{B}(r_{m,n},3)$.
If $N(u_0)\subseteq V(H^1)$, then $G^1\cong K_{r_{m,n}}^{t_{m,n}}$.
And since $\delta(K_{r_{m,n}}^{t_{m,n}})\geq s-1$, $G^{s-2}\cong G^1\cong K_{r_{m,n}}^{t_{m,n}}.$

\noindent{\bf{Subcase 2.2.}} $d_0=r_{m,n}-1\geq t_{m,n}+1$.

Now $r_{m'-d_0,n'-1}=r_{m,n}-1$ and $t_{m'-d_0,n'-1}=t_{m,n}$. Since $d_0=r_{m,n}-1\geq s-1$, the first inequality of (6) holds in equality if and only if $N(u_0)$ is a clique. The second inequality of (6) holds in equality if and only if $k_s(H)={{r_{m,n}-1}\choose s}+{{t_{m,n}}\choose {s-1}}$.

\noindent{\bf{Subcase 2.2.1.}} $t_{m,n}\le s-2$.

Note that $N[u_0]$ is an $r_{m,n}$-clique. Let $H_0=G[N[u_0]].$ Then
$|E({H_0})|-|V({H_0})|={r_{m,n}\choose 2}-r_{m,n}={r_{m,n}-1\choose 2}-1$. Combining with equality (1), we have $[|E(G)|-|V(G)|]-[|E(H_0)|-|V(H_0)|]=t_{m,n}-1\le s-3$.
By Lemma \ref{aaa}, we have $G^{s-2}\cong H_0^{s-2}\cong K_{r_{m,n}}$.

\noindent{\bf{Subcase 2.2.2.}} $t_{m,n}\ge s-1$.

Then $k_s(H)>0$ and hence ${s\choose 2}-s\leq(m'-d_0)-(n'-1)<m-n$. By the induction hypothesis, $H$ meets (ii), that is, $H^{s-2}\cong K_{r_{m,n}-1}^{t_{m,n}}$, or $H^{s-2}\in\mathbb{B}(r_{m,n}-1,3)$ for $t_{m,n}=s-1=2$. In both cases, one can find that $|E(H^{s-2})|-|V(H^{s-2})|={{r_{m,n}-1}\choose 2}+t_{m,n}-r_{m,n}$. And by equality (2), $|E(H)|-|V(H)|={{r_{m,n}-1}\choose 2}+t_{m,n}-r_{m,n}$.
So $|E(H)|-|V(H)|=|E(H^{s-2})|-|V(H^{s-2})|$. By Lemma~\ref{aaa}, $H^1\cong(H^{s-2})^1\cong H^{s-2}$. Since $N(u_0)$ is an $(r_{m,n}-1)$-clique, where $r_{m,n}-1\geq t_{m,n}+1\geq s\geq3$, we have $N(u_0)\subseteq V(H^1)$. This indicates that $G^1\cong K_{r_{m,n}}^{t_{m,n}}$, or $G^1\in \mathbb{B}(r_{m,n},3)$ for $t_{m,n}=s-1=2$ and $r_{m,n}\geq s=3$. Since in both cases $G^{s-2}\cong G^1$, we are done.

Conversely, for any graph $G\in \mathscr{G}_{m,n}$ described in (\romannumeral1) or (\romannumeral2), it is easy to see that $k_s(G)=k_s(G^{s-2})={{r_{m,n}}\choose s}+{{t_{m,n}}\choose {s-1}}$.
This completes the proof.
\end{proof}

\section{Extremal graph on spectral moment}

For a given graph $H$, a subgraph of $G$ isomorphic to $H$ is called an $H$-subgraph of $G$. Denote by $\phi_G(H)$ (or $\phi(H)$) the number of $H$-subgraphs in $G$. In this section, we will determine the last graph in $S$-order over all connected graphs of size $m$ and order $n$.
First, we need to give some basic lemmas.

\begin{lemma}\label{22}( D.Cvetkovi\'{c}, M.Doob and H.Sachs\cite{aa}) For any graph $G$, $ S_4(G)=2\phi(P_2)+4\phi(P_3)+8\phi(C_4)$, where $S_j(G)$ is the number of closed walks of length $j$.
\end{lemma}

A nonincreasing sequence $(d_1,d_2,\ldots,d_n)$ is denoted by $\pi_G$, if it is a degree sequence of a graph $G$.

\begin{lemma}\label{1}
Let $G$ be a connected graph of order $n$ with $\pi_{G^1}=(\bar{d_1},\bar{d_2},\cdots,\bar{d_k})$.
If $n>k$ and a nonincreasing positive sequence $(d_1,d_2,\cdots,d_n)$ satisfies\\
(\romannumeral1) $d_i\ge\bar{d_i}$ for $1\le i\le k$ and $d_{i_0}>\bar{d_{i_0}}$ for some $i_0\in\{1,2,\cdots,k\}$,\\
(\romannumeral2) $\sum_{i=1}^{n}d_i=\sum_{i=1}^{k}\bar{d_i}+2(n-k)$,\\
Then $\pi$ is graphic. Specially, there exists a connected graph $G^*$ such that $G^{*1}\cong G^1$ and $\pi(G)=(d_1,d_2,\cdots,d_n).$
\end{lemma}

\begin{proof}
Let $s$ be the maximum number in $\{0,1,\ldots,n-k\}$ such that $d_{k+s}\ge2$. Put $d_{G^1}(u_i)=\bar{d_i}$ for each $i\in \{1,2,\cdots,k\}$.
We construct a new graph $G^{*}$ as follows. Let $G'$ be a graph obtained from $G^1$ and a path $P_s=u_{k+1}u_{k+2}\cdots u_{k+s}$ by adding an edge $u_{i_0}u_{k+1}$ for $s\ge 1$, or $G'\cong G^1$ for $s=0$. Let $G^*$ be the graph obtained from $G'$ by adding $d_i-d_{G'}(u_i)$ pendant edges to $u_i$ for each $i\in\{1,2,\cdots,k+s\}$. In this way, we obtain a connected graph $G^{*}$ with $d_{G^*}(u_i)=d_i$ for $1\le i \le k+s$. The number of pendant vertices in $G^*$ is
\begin{eqnarray*}
\sum_{i=1}^{k+s}[d_i-d_{G'}(u_i)]
&=& \sum_{i=1}^{k+s}{d_i}-\sum_{i=1}^{k+s}{d_{G'}(u_i)}\\
&=& [\sum_{i=1}^{n}{d_i}-(n-k-s)]-[\sum_{i=1}^{k}\bar{d_i}+2s]\\
&=& n-k-s,
\end{eqnarray*}
since $\sum_{i=1}^{n}d_i=\sum_{i=1}^{k}\bar{d_i}+2(n-k)$. So $|V(G^*)|=n$. According to the definition, it is easy to see that $G^{*1}\cong G^1$ and $\pi(G)=(d_1,d_2,\cdots,d_n)$.
\end{proof}

\begin{lemma}\label{0}
Let $(\bar{d_1},\bar{d_2},\cdots,\bar{d_k})$ be a nonincreasing sequence. Let $n\ge k$ and $\pi'=(d_1',d_2',\cdots,d_n')$ be a sequence with $d_i'\ge\bar{d_i}$ for $1\leq i\leq k$. Reorder $\pi'$ in a nonincreasing order $\pi=(d_1,d_2,\cdots,d_n)$. Then we also have $d_i\ge \bar{d_i}$ for $1\leq i\leq k$.
\end{lemma}

\begin{proof} Given any $i\in\{1,2,\ldots,k\}$. For each $t\in \{1,2,\cdots,i\}$, we have $d_t'\ge \bar{d_t}\ge \bar{d_i}$.
This implies that there are at least $i$ elements not less than $\bar{d_i}$ in $\pi'$.
Since $d_i$ is the $i$-th largest element in $\pi$ (also in $\pi'$), we have $d_i\ge \bar{d_i}$.
\end{proof}

\begin{lemma}\label{00}
Let $G\in\mathscr{G}_{m,n}$ with $\pi(G)=(d_1,d_2,\cdots,d_{n})$ and
$\pi_{G^1}=(\bar{d_1},\bar{d_2},\cdots,\bar{d_k})$. If $n>k$ and $d_{i_0}>\bar{d}_{i_0}$ for some $i_0\in\{2,3,\ldots,k\}$, then there exists a graph $G^*\in\mathscr{G}_{m,n}$ such that ${G^*}^1\cong G^1$ and $S_4(G^*)>S_4(G)$.
\end{lemma}

\begin{proof} Note that $G^1$ is an induced subgraph of $G$. Although the vertex of degree $\bar{d_i}$ in $G^1$ may not be the vertex of degree $d_i$ in $G$, by Lemma~\ref{0}, we still have $d_i\ge \bar{d_i}$ for $1\leq i\leq k$.

Now we define a sequence $\pi'=(d_1',d_2',\cdots,d_{n}')$, where $d_{1}'=d_{1}+1$, $d_{i_0}'=d_{i_0}-1$ and $d_i'=d_i$ for each $i\notin\{1,i_0\}$. Clearly, $d_i'\geq \bar{d_i}$ for $1\leq i\leq k$. Reorder $\pi'$ in a nonincreasing order $\pi''=(d_1'',d_2'',\cdots,d_{n}'')$. By Lemma~\ref{0}, $d_1''=d_{1}'>d_1$ and $d_i''\ge \bar{d_i}$ for $2\leq i\leq k$.
Moreover, since $G^1$ is obtained from $G$ by iteratively deleting $n-k$ pendant edges, we have $\sum_{i=1}^nd_i''=\sum_{i=1}^nd_i=\sum_{i=1}^k\bar{d_i}+2(n-k).$
By Lemma~\ref{1}, there exists a connected graph $G^*$ such that ${G^*}^1\cong G^1$ and $\pi(G^*)=\pi''$. This indicates that $G^*\in\mathscr{G}_{m,n}$.

Furthermore, it is clear that $\phi_{G^*}(P_2)=\phi_{G}(P_2)$ and $\phi_{G^*}(C_4)=\phi_{G}(C_4)$. Note that $\phi_{G}(P_3)=\sum_{i=1}^n{d_i\choose 2}$. So $\phi_{G^*}(P_3)-\phi_{G}(P_3)=d_{1}-d_{i_0}+1>0$, since $1<i_0$. By Lemma \ref{22}, we have $S_4(G^*)>S_4(G)$.
\end{proof}

For convenience, we call $\pi''$, in the proof of Lemma \ref{00}, a $d_{i_0}$-\emph {transformation} of $\pi$. Correspondingly, we call $G^*$ a $d_{i_0}$-\emph {transformation} of $G$. We also write $\varsigma (H,n)$ for the set of connected graphs $G$ of order $n$ with $G^1\cong H$. Since $G^1$ is obtained from $G$ by iteratively deleting $n-k$ pendant edges, one can see that all graphs in $\varsigma (H,n)$ have the same size $|E(H)|+(n-k)$.

\begin{lemma}\label{12}
Let $n>k$ and $H$ be a connected graph with $\pi_{H}=(\bar{d_1},\bar{d_2},\cdots,\bar{d_k})$, where $\bar{d_k}>1$. Then $G$ attains the largest value of $S_4$ over all graphs in $\varsigma(H,n)$ if and only if $\pi_{G}=(d_1,d_2,\cdots,d_{n})$, where $d_1=\bar{d_1}+n-k$, $d_i=\bar{d_i}$ for $2\leq i\leq k$ and $d_i=1$ for $i>k$.
\end{lemma}

\begin{proof} Let $G\in\varsigma(H,n)$ be the extremal graph and $\pi_{G}=(d_1,d_2,\cdots,d_{n})$. Since $H$ is an induced subgraph of $G$, by Lemma \ref{0}, $d_i\geq\bar{d_i}$ for $1\leq i\leq k$. Suppose that $d_{i_0}>\bar{d_{i_0}}$ for some $i_0\in\{2,3,\dots,k\}$. Let $G^*$ be a $d_{i_0}$-transformation of $G$. Then by Lemma \ref{00}, ${G^*}^1\cong G^1\cong H$ and $S_4(G^*)>S_4(G)$, a contradiction.
Now suppose that $d_{i_0}>1$ for some $i_0>k$,
say, $u\in V(G)$ corresponding to $d_{i_0}$ and $v\in N_G(u)$ with minimal distance to $G^1$. We define $G^*=G-uv+uw$, where $w\in V(G)$ corresponding to $d_1$. Clearly, ${G^*}^1\cong G^1$ and by Lemma \ref{22},
$$S_4(G^*)-S_4(G)=4[\phi_{G^*}(P_3)-\phi_{G}(P_3)]=4(d_1-d_{i_0}+1)>0,$$ also a contradiction. So $d_i=\bar{d_i}$ for $2\leq i\leq k$ and $d_i=1$ for $k+1\leq i\leq n$. Recall that $\sum_{i=1}^nd_i=\sum_{i=1}^k\bar{d_i}+2(n-k).$ Thus $d_1=\bar{d_1}+n-k$.

Conversely, assume that $G',G''\in\varsigma(H,n)$ with $\pi_{G'}=\pi_{G''}=(d_1,d_2,\cdots,d_{n})$. Since ${G'}^1\cong{G''}^1\cong H$ and $\pi_{G'}=\pi_{G''}$, we have $\phi_{G'}(P_2)=\phi_{G''}(P_2)$ and $\phi_{G'}(C_4)=\phi_{G''}(C_4)$. Moreover, note that $\phi_{G'}(P_3)=\phi_{G''}(P_3)=\sum_{i=1}^n{d_i\choose 2}$. By Lemma \ref{22}, $S_4(G')=S_4(G'')$.
\end{proof}

Now we try to characterize the last graph in $S$-order over all graphs in $\mathscr{G}_{m,n}$.

\begin{theorem}\label{105} Let $G^{*}$ be the last graph in $S$-order over all graphs in $\mathscr{G}_{m,n}$. Then $G^{*}$ is obtained from $K_{r_{m,n}}^{t_{m,n}}$ by adding ${n-r_{m,n}-1}$ pendant edges to a vertex of degree $r_{m,n}$.
\end{theorem}

\begin{proof} It is known that $\sum_{i=1}^n\lambda_i(G)=0$,
$\sum_{i=1}^n\lambda_i^2(G)=2|E(G)|$ and $\sum_{i=1}^n\lambda_i^3(G)=6k_3(G)$ for a graph $G$
(see \cite{aa}). So $S_1(G)=0$, $S_2(G)=2m$ and $S_3(G)=6k_3(G)$ for any graph $G\in\mathscr{G}_{m,n}$. And since $G^*$ is the last graph in $S$-order, $G^*$ has maximal number of triangles over all graphs in $\mathscr{G}_{m,n}$.

If $t_{m,n}\geq 3$, then by Theorem~\ref{222} (ii), ${G^*}^1\cong  K_{r_{m,n}}^{t_{m,n}}$, that is, $G^*\in\varsigma(K_{r_{m,n}}^{t_{m,n}},n)\subseteq \mathscr{G}_{m,n}$. Since $S_3(G)=S_3(K_{r_{m,n}}^{t_{m,n}})$ for any graph $G\in\varsigma(K_{r_{m,n}}^{t_{m,n}},n)$, $G^*$ must attain the largest value of $S_4$ over all graphs in $\varsigma(K_{r_{m,n}}^{t_{m,n}},n)$.
By Lemma \ref{12}, one can find that the statement holds.

If $t_{m,n}=2$, then by Theorem~\ref{222} (ii), $G^{*1}\cong  K_{r_{m,n}}^{2}$, or $G^{*1}\in \mathbb{B}(r_{m,n},3)$.
For any graph $G$ with $G^1 \in \mathbb{B}(r_{m,n},3)$,
if there is a cut edge $uv$ of $G^1$, let $G'$ be the graph obtained from $G-uv$ by identifying $u$ with $v$ and adding a new pendant edge $uw$. Note that $S_3(G)=S_3(G')$ and $S_4(G)<S_4(G')$. This
implies that if $G^{*1}\in \mathbb{B}(r_{m,n},3)$ then $G^{*1}\cong B(r_{m,n},3)$. Furthermore, by Lemma~\ref{12}, $ G^{*}\cong B_1 $ if $G^{*1}\cong K_{r_{m,n}}^{2}$ or $ G^{*}\cong B_2 $ if $G^{*1}\cong B(r_{m,n},3)$ (see Fig.~\ref{bbc}). Clearly, $S_3(B_1)=S_3(B_2)$. And by Lemma~\ref{22} and direct computations, we have $S_4(B_2)<S_4(B_1)$. Hence $G^* \cong B_1$, the statement also holds. This completes the proof.
\end{proof}

\begin{figure}[!ht]
	\centering
	\includegraphics[width=0.5\textwidth]{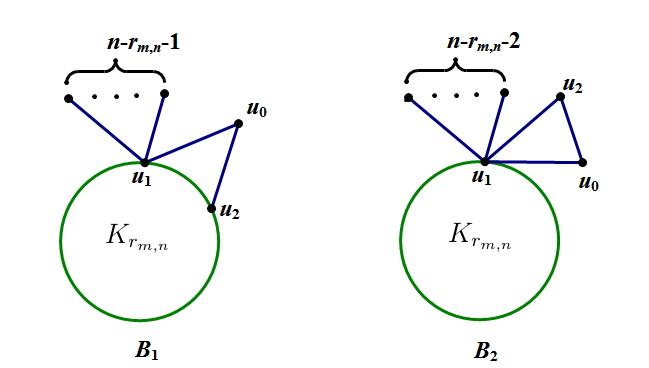}
	\caption{Extremal graphs $B_1$ and $B_2$ }{\label{bbc}}
\end{figure}

%\vspace{6bp}
%\noindent{\bf Acknowledgements} We would like to show our gratitude to anonymous referees for their valuable suggestions which largely improve the quality of this paper.

\small {

}
\end{document}